\documentclass[11pt]{article}

\usepackage{latexsym,amssymb,upref,amsmath,amsthm, amsfonts,authblk}
\usepackage{amssymb,amsmath,amsthm, calc, graphicx}
\usepackage{epsfig}
\usepackage{breqn}
\usepackage{footnpag}
\usepackage{rotating}
\usepackage{amsfonts}
\usepackage{setspace}
\usepackage{fullpage}
\usepackage{enumitem}
\usepackage{bbold}
\usepackage{comment}
\usepackage{pgf,tikz}
\usepackage{mathrsfs}
\usetikzlibrary{arrows}
\usepackage{yhmath}
\usepackage{hyperref}
\usepackage{authblk}

\newtheorem{thm}{Theorem}
\newtheorem{definition}{Definition}
\newtheorem{claim}[thm]{Claim}
\newtheorem{lemma}[thm]{Lemma}
\newtheorem{corollary}[thm]{Corollary}
\newtheorem{problem}[thm]{Problem}

\newtheorem{remark}[thm]{Remark}

\newtheorem{observation}[thm]{Observation}
\newtheorem*{cons}{Construction}
\newtheorem{case}{Case}

\newcommand{\thistheoremname}{}
\newtheorem*{genericthm*}{\thistheoremname}
\newenvironment{namedthm*}[1]
{\renewcommand{\thistheoremname}{#1}%
	\begin{genericthm*}}
	{\end{genericthm*}}

\renewcommand{\P}{\mathcal{P}}

\newcommand{\abs}[1]{\left\lvert{#1}\right\rvert}

\title{$3$-uniform hypergraphs and linear cycles}
\author{
	Beka Ergemlidze\thanks{
		Department of Mathematics, Central European University, Budapest.
		E-mail: beka.ergemlidze@gmail.com
	} \qquad Ervin Gy\H{o}ri \thanks{R\'enyi Institute, Hungarian Academy of Sciences and
	Department of Mathematics, Central European University, Budapest. E-mail: gyori.ervin@renyi.mta.hu} \qquad Abhishek Methuku \thanks{Department of Mathematics, Central European University, Budapest. E-mail: abhishekmethuku@gmail.com}
}

\begin{document}

\maketitle

\begin{abstract}
Gy\'arf\'as, Gy\H{o}ri and Simonovits \cite{GyGyoSim} proved that if a $3$-uniform hypergraph with $n$ vertices has no linear cycles, then its independence number $\alpha \ge \frac{2n} {5}$. The hypergraph consisting of vertex disjoint copies of a complete hypergraph $K_5^3$ on five vertices shows that equality can hold. They asked whether this bound can be improved if we exclude $K_5^3$ as a subhypergraph and whether such a hypergraph is $2$-colorable.

In this paper, we answer these questions affirmatively. Namely, we prove that if a $3$-uniform linear-cycle-free hypergraph doesn't contain $K_5^3$ as a subhypergraph, then it is $2$-colorable. This result clearly implies that its independence number $\alpha \ge \lceil \frac{n}{2} \rceil$. We show that this bound is sharp.

Gy\'arf\'as, Gy\H{o}ri and Simonovits also proved that a linear-cycle-free $3$-uniform hypergraph contains a vertex of strong degree at most 2. In this context, we show that a linear-cycle-free $3$-uniform hypergraph has a vertex of degree at most $n-2$ when $n \ge 10$.

%
%

\end{abstract}

\section{Introduction}

A $3$-uniform hypergraph $H = (V, E)$ consists of a set of vertices $V$ and a set of hyperedges $E$ such that each hyperedge is a $3$ element subset of $V$.
$H$ is $k$ colorable if there is a coloring of the vertices of $H$ with $k$ colors such that there is no monochromatic hyperedge in $H$. Throughout the paper, we mostly use the terminology introduced in \cite{GyGyoSim}.

\begin{definition}
	A \emph{linear tree} is a hypergraph obtained from a vertex by repeatedly adding hyperedges that intersect the previous hypergraph in exactly one vertex. 
	
	
	 A linear path $\mathcal P$ of length $k \ge 0$ is an alternating sequence $v_1,h_{1},v_2,h_{2},...,h_k, v_{k+1}$ of distinct vertices and distinct hyperedges such that $h_i \cap h_{i+1} = \{v_{i+1}\}$ for each $i \in \{1,2,\ldots,k-1\}$, $v_1 \in h_1, v_{k+1} \in h_k$ and $h_i \cap h_{j}= \emptyset$ if $\abs{j-i} \ge 2$. The vertex set $V(\mathcal P)$ of $\mathcal P$ is $\cup^k_{i=1} h_i$ or $\{v_1\}$ if $k = 0$.
	 
	  We say that $\mathcal P$ is a linear path between/joining $v_1$ and $v_{k+1}$ or in general, between vertex sets $A$ and $B$ if $v_1 \in A$, $v_{k+1} \in B$, $h_i \cap A = \emptyset$ for $2 \le i \le k$ and $h_i \cap B = \emptyset$ for $1 \le i \le k-1$. Typically $A$ and $B$ are (vertex sets of) hyperedges or one element sets.
	
	A \emph{linear cycle} of length $k \ge 3$ is an alternating sequence $v_1,h_{1},v_2,h_{2},...,v_{k}, h_k$ of distinct vertices and distinct hyperedges such that $h_i \cap h_{i+1} = \{v_{i+1}\}$ for each $i \in \{1,2,\ldots,k-1\}$, $h_1 \cap h_k = \{v_1\}$ and $h_i \cap h_{j}= \emptyset$ if $1 < \abs{j-i} < k-1$.
	
	A \emph{skeleton} $T$ in $H$ is a linear subtree of $H$ which cannot be extended to a larger linear subtree by adding a hyperedge $e$ of $H$ for which $\abs{e \cap V(T)} = 1$.
\end{definition}

An independent set in $H$ is a set of vertices containing no hyperedge of $H$. More precisely, if $I$ is an independent set of $H$, then there is no $e \in E(H)$ such that $e \subseteq I$. Let $\alpha(H)$ denote the size of the largest independent set in $H$. Gy\'arf\'as, Gy\H{o}ri and Simonovits \cite{GyGyoSim} initiated the study of linear-cycle-free hypergraphs by showing:

\begin{thm}(Gy\'arf\'as, Gy\H{o}ri, Simonovits \cite{GyGyoSim})
	If $H$ is a $3$-uniform hypergraph on $n$ vertices without linear cycles, then it is $3$-colorable. Moreover, $\alpha(H) \ge \frac{2n}{5}$.
\end{thm}

If the hypergraph does not contain the complete $3$-uniform hypergraph $K_5^3$ as a subhypergraph then a stronger theorem can be proved, answering a question of Gy\'arf\'as, Gy\H{o}ri and Simonovits.

\begin{thm}
\label{2coloring}
If a $3$-uniform linear-cycle-free hypergraph $H$ doesn't contain $K_5^3$ as a subhypergraph, then it is $2$-colorable.
\end{thm}

\begin{corollary}
If a $3$-uniform linear-cycle-free hypergraph $H$ on $n$ vertices doesn't contain $K_5^3$ as a subhypergraph, then $\alpha(H) \ge \lceil \frac{n}{2} \rceil$ and the bound is sharp.
\end{corollary}

Indeed, from Theorem \ref{2coloring}, it trivially follows that $\alpha(H) \ge \lceil \frac{n}{2} \rceil$. The hypergraph $H_n$ on $n$ vertices obtained from the following construction shows that this inequality is sharp. Let $H_3$ be the hypergraph on $3$ vertices $v_1, v_2, v_3$ such that $v_1v_2v_3 \in E(H_3)$ and let $H_4$ be the complete $3$-uniform hypergraph $K_4^3$ on $4$ vertices $v_1, v_2, v_3, v_4$. Now for each $3 \le i \le n-2$ let us define the hypergraph $H_{i+2}$ such that $V(H_{i+2}) := V(H_{i}) \cup \{v_{i+1}, v_{i+2} \}$ and  $E(H_{i+2}):= E(H_i) \cup \{v_{i+1} v_{i+2} v_j\}_{j=1}^i $. If $n$ is even, we start this iterative process with the hypergraph $H_4$ and if $n$ is odd, we start with $H_3$. Notice that $\alpha(H_{i+2}) = \alpha(H_i)+1$ for each $i$, which implies that $\alpha(H_n) = \lceil \frac{n}{2} \rceil$.

It is another natural problem to bound the number of hyperedges or different types of degrees of vertices in hypergraphs with no linear cycles. The most plausible is the \emph{degree} of a vertex $v \in V$ what is simply the number of hyperedges of $H$ containing $v$. Given a $3$-uniform hypergraph $H$ and $v \in V(H)$,  the \emph{link} of $v$ in $H$ is the graph with vertex set $V(H)$ and edge set $\{xy : vxy \in E(H)\}$. The \emph{strong degree} $d^+(v)$ of $v \in V(H)$ is the maximum number of independent edges in the link of $v$. It is interesting and known for many years that the maximal number of hyperedges in a $3$-uniform hypergraph without linear cycles is ${n-1 \choose 2}$, which is the maximum number of hyperedges without a linear triangle \cite{CSK, FF}.
The relation to the strong degree is proved recently.

\begin{thm}(Gy\'arf\'as, Gy\H{o}ri, Simonovits \cite{GyGyoSim})
	Let $H$ be a $3$-uniform hypergraph without linear cycles. Then, it has a vertex $v$ whose strong degree $d^+(v)$ is at most $2$.
\end{thm}

In this paper, we show a similar and perhaps more natural theorem concerning the degree of a linear-cycle-free hypergraph.

\begin{thm}
\label{degree_condition}
Let $H$ be a $3$-uniform hypergraph on $n \ge 10$ vertices without linear cycles. Then, there is a vertex whose degree is at most $n-2$.
\end{thm}

\begin{remark}
\label{9vertex}
There is a $3$-uniform hypergraph on $9$ vertices without linear cycles where the degree of every vertex is $8$. This hypergraph $H$ is defined by taking a copy of $K_4^3$ on vertices $\{u_1, u_2, v_1, v_2\}$ and a vertex disjoint copy of $K_5^3$ such that $u_1u_2x, v_1v_2x \in E(H)$ for each $x \in V(K_5^3)$ and there are no other hyperedges in $H$.
\end{remark}

\begin{remark}
\label{degreesharpness}
Theorem \ref{degree_condition} cannot be improved because there is a $3$-uniform hypergraph $H'$, with $E(H') := \{ xab\mid a,b \in V(H')\setminus \{x\} \}$ for a fixed vertex $x \in V(H)$, in which every vertex has degree at least $n-2$.
\end{remark}
The paper is organized as follows: In section \ref{notationsection} we introduce some important definitions. In section \ref{theorem2} we prove Theorem \ref{2coloring} by means of our main lemma - Lemma \ref{skeleton_coloring} (which is proved in section \ref{mainlemma}). In section \ref{degree} we prove Theorem \ref{degree_condition}. Finally in section \ref{concluding_remarks}, we present some concluding remarks and open questions.

\section{Definitions}
\label{notationsection}
The following notions of \emph{association} are used throughout the paper.

\begin{definition}
	Given a vertex $v \in V(H)$ and a hyperedge $abc \in E(H)$ such that $ v \not \in\{a,b,c\}$, we say that $v$ is \emph{``strongly associated"} to $abc$ if at least two of the three edges $vab$, $vbc$, $vca$ are in $E(H)$ . We say that $v$ is \emph{``weakly associated"} to $abc$ if exactly one of the three edges $vab$, $vbc$, $vca$ is in $E(H)$. We say that $v$ is associated to $abc$ if it is either strongly or weakly associated.
	
	The set of pairs $\{\{x,y\} \subset \{a,b,c\} \mid vxy \in E(H) \}$ is called the ``support" of $v$ in $abc$, denoted $s_{abc}(v)$ and these hyperedges $vxy$ are called  ``supporting" hyperedges of $v$ in $abc$.
\end{definition}

\begin{remark}
	
	The main motivation for the above definition is the following fact. If $P$ is a linear path ending in a hyperedge $abc$ and $v \not \in V(P)$ is a vertex strongly associated to $abc$ then $P$ can be extended by one of the supporting hyperedges of $v$ in $abc$ to a longer linear path.

\end{remark}

\section{Proof of Theorem \ref{2coloring}: $2$-colorability of linear-cycle-free hypergraphs containing no $K^3_5$}
\label{theorem2}
Let $H$ be a $3$-uniform hypergraph without linear cycles.

\begin{claim}
	\label{ass}
If $T$ is a linear tree and $v \in V(T)$ such that $v$ is strongly associated to a hyperedge $abc$ of $T$, then $v$ belongs to a hyperedge of $T$ neighboring (not disjoint to) $abc$. If $v \not \in V(T)$, and $v$ is strongly associated to $h_1, h_2 \in E(T)$ then $h_1$ and $h_2$ are neighboring hyperedges. 
\end{claim}
\begin{proof}
To prove the first statement of the claim, suppose that $v$ is not in a neighboring hyperedge of $abc$. Then, take the linear path (of length at least $2$) from $v$ to $abc$ in $T$ and the appropriate supporting hyperedge of $v$ in $abc$ to produce a linear cycle, a contradiction. To prove the second statement, suppose that $h_1$ and $h_2$ are not neighboring hyperedges. Then, take the linear path (of length at least $1$) in $T$ joining $h_1$ and $h_2$ and an appropriate supporting hyperedge of $v$ in $h_1$ and $h_2$ respectively to produce a linear cycle, a contradiction.
\end{proof}

\begin{definition} [thick pair, thick hyperedge]
For any two vertices, $a, b \in V(H)$, we call the pair $\{a, b\}$  ``thick" if there are at least two different hyperedges containing $\{a, b\}$. We call a hyperedge $abc$ ``thick" if all the pairs $\{a, b\}$, $\{b, c\}$ and $\{c, a\}$ are thick.
\end{definition}

\begin{lemma}
	\label{association}
If $abc \in E(H)$ is a thick hyperedge, then the set of vertices associated to it consists of one of the following
\begin{enumerate}
\item Two vertices that are strongly associated to $abc$ (and no vertices that are weakly associated to $abc$). 
\item One vertex that is strongly associated to $abc$ and vertices $w_1, w_2, \ldots, w_m$ such that each $w_i$ is weakly associated to $abc$ and $\abs{\cup_{i} s_{abc}(w_i)} = 1$. (It is possible that $m=0$, i.e., no such $w_i$ exists).
\end{enumerate}
\end{lemma}

\begin{proof}
	
If there is no vertex strongly associated to $abc$, then since $abc$ is thick, we must have $3$ distinct vertices $v_1, v_2, v_3$ such that $v_1ab, v_2bc, v_3ca \in E(H)$, a linear cycle, a contradiction. So there must be a vertex strongly associated to $abc$.

Now we show that if there are two vertices $p, q$ strongly associated to a hyperedge $abc \in E(H)$, then there are no other vertices associated to $abc$. Suppose by contradiction that there are such vertices. Then, among these vertices there is a vertex $r$ such that $\abs{s_{abc}(p) \cup s_{abc}(q) \cup s_{abc}(r)} = 3$ since $abc$ is thick.
Now consider the bipartite graph whose two color classes are $\{p,q,r\}$ and $\{\{a,b\}, \{b,c\}, \{c,a\}\}$ where $v \in \{p,q,r\}$ is connected to $\{x,y\} \in \{\{a,b\}, \{b,c\}, \{c,a\}\}$ if $vxy \in E(H)$. It can be easily checked that Hall's condition holds for the color class $\{p,q,r\}$ and so there exists a matching between the two color classes, but this corresponds to a linear cycle (of size $3$) in $H$, a contradiction.

So the only remaining possibility is that $abc$ has exactly one vertex which is strongly associated to it and maybe some other vertices $w_1, w_2, \ldots, w_m$ that are weakly associated to it. We only have to show that $\abs{\cup_{i} s_{abc}(w_i)} = 1$. Suppose by contradiction that there are vertices $w_i$ and $w_j$ such that their supports in $abc$ are different. Let $s_{abc}(w_i) = \{\{a,b\}\}$ and $s_{abc}(w_j) = \{\{b,c\}\}$ without loss of generality. Then, since $abc$ is thick, there is a vertex $v$ such that $v \not = w_i$, $v \not = w_j$ and $acv \in E(H)$. Now, $acv$, $abw_i$, $bcw_j$ is a linear cycle, a contradiction.
\end{proof}

%

%
Given a set of vertices $S \subseteq V(H)$, the subhypergraph of $H$ induced by $S$ is defined as a hypergraph whose vertex set is $S$ and edge set is $\{e \in E(H) \mid e \subseteq S\}$.

\begin{lemma}[Main Lemma]
\label{skeleton_coloring}
Let $T$ be a linear tree in $H$. Then there exists a $2$-coloring $\gamma: V(T) \mapsto \{1,2\}$, such that the following properties hold:

\begin{enumerate}
	\item The subhypergraph induced by $V(T)$ is properly $2$-colored.
	
	\item For each vertex $v \in V(H) \setminus V(T)$ that is strongly associated to some hyperedge of $T$, $v$ can be colored (by color $1$ or $2$) so that all hyperedges $vab$ with $a, b \in V(T)$ are properly $2$-colored.
	
	\item For each remaining vertex $v \in V(H) \setminus V(T)$, all the hyperedges $vab$ with $a, b \in V(T)$ satisfy the property $\gamma(a) \not = \gamma(b)$ (i.e., these hyperedges $vab$ are properly $2$-colored regardless of how we fix the color of $v$ later).
\end{enumerate}

\end{lemma}

Before we prove this lemma, we will show how to prove Theorem \ref{2coloring} using it.

\begin{observation}
	\label{observation}
Let $w \in V(T)$. Notice that the above lemma holds even if we add the extra condition that the color of $w$ is given.
\end{observation}

Now we prove our main theorem using this lemma.

\begin{proof}[Proof of Theorem \ref{2coloring}]
Let $T_1$ be any skeleton of $H$.  Then there exists a $2$-coloring of $T_1$ given by Lemma \ref{skeleton_coloring}. Let $U_1 \subseteq V(H)  \setminus V(T_1)$ be the set of all vertices such that each $u \in U_1$ is strongly associated to some hyperedge of $T_1$. If $\abs{U_1} = 0$, then by Lemma \ref{skeleton_coloring} all the vertices of $V(H) \setminus V(T_1)$ can be $2$-colored arbitrarily such that the hyperedges $vab$ with $a, b \in V(T_1)$ are properly $2$-colored. Also, since $T_1$ is a skeleton, there are no hyperedges $vxy$ where $v \in V(T_1)$ and $x, y \in V(H) \setminus V(T_1)$. Therefore, the vertices of $V(H) \setminus V(T_1)$ can be $2$-colored independently from vertices of $V(T_1)$ and so we have the same problem for the subhypergraph induced by $V(H) \setminus V(T_1)$. So we can assume that $\abs{U_1} \not = 0$. Now let us define a sequence of linear trees $T_1, T_2, \ldots, T_i, T_{i+1}, \ldots, T_m$ recursively as follows: Let $U_{i} \subseteq V(H)  \setminus \cup_{j=1}^i V(T_j)$ be the set of vertices where each $u \in U_{i}$ is strongly associated to some hyperedge of $\cup_{j=1}^i T_j$ and let $T_{i+1}$ be a skeleton in the subhypergraph induced by $V(H)  \setminus \cup_{j=1}^i V(T_j)$ so that $T_{i+1}$ contains at least one vertex from $U_{i}$ (we continue this procedure as long as $\abs{U_i} \not = 0$; so $\abs{U_m} = 0$). Notice that $T_{i+1}$ might consist of just one vertex. In fact, we will show that $\abs{V(T_{i+1}) \cap U_i} = 1$. Let $H_i$ denote the subhypergraph of $H$ induced by $\cup_{j=1}^i V(T_j)$.

\begin{claim}
\label{TandU}
For each $1 \le i \le m-1$, there is a linear path in $H_i$ between any two vertices $u, v \in V(H_i)$. Moreover, $V(T_{i+1}) \cap U_i$ consists of only one vertex and this vertex can be strongly associated to hyperedge(s) of $T_s$ for exactly one $s$, $1 \le s \le i$.
\end{claim}

\begin{proof}[Proof of Claim \ref{TandU}]
We prove the claim by induction on $i$. For $i = 1$, the statement is trivial. Assume the statement is true for $i = k$.  First we will show that there is a linear path between $u \in V(T_{k+1}) \cap U_k$ and any $v \in V(H_k)$. Let $abc \in E(T_s)$ (for some $1 \le s \le k$) be the hyperedge in $\cup_{j=1}^k T_j$ such that $u$ is strongly associated to $abc$. Consider a linear path $\P_1$ in $H_k$ between $v$ and $\{a,b,c\}$ (in case, $v \in \{a,b,c\}$, $\P_1$ consists of just $v$).
By adding an appropriate supporting hyperedge of $u$ in $abc$, $\P_1$ is extended to a linear path between $u$ and $v$.
 Notice that this path contains only one vertex from $T_{k+1}$. Since there is a linear path between every $2$ vertices of $T_{k+1}$ we have a linear path between any vertex of $T_{k+1}$ and any vertex of $H_k$. By the induction hypothesis there is a linear path between any two vertices of $H_k$ and so we have proved the first part of the claim.

Now assume by contradiction that there is a vertex $u' \not = u$, $u' \in V(T_{k+1}) \cap U_k$ which is strongly associated to a hyperedge $pqr \in \cup_{j=1}^k E(T_j)$. Take a linear path $\P_2$ in $H_k$ between $\{a,b,c\}$ and $\{p,q,r\}$. Extend $\P_2$ on both ends by appropriate supporting hyperedges of $u$ in $abc$ and $u'$ in $pqr$ respectively. Then this path together with the linear path  in $T_{k+1}$ between $u$ and $u'$ is a linear cycle, a contradiction.


So $V(T_{k+1}) \cap U_k$ consists of only vertex, say $u$. If $u$ is strongly associated to two hyperedges $h_1 \in T_r$ and $h_2 \in T_s$ (where $r \not= s$ and $r, s \le k$), then take a linear path $\P$ in $H_k$ between $h_1$ and $h_2$ and extend it by appropriate supporting hyperedges of $u$ in $h_1$ and $h_2$ to a linear cycle, a contradiction. 
\end{proof}

We will show that for each $1 \le k \le m$, $H_k$ is properly $2$-colored such that each $T_i$, $i \le k$ is $2$-colored according to Lemma \ref{skeleton_coloring}. For $k =1$ the above statement is trivially true. Let us assume that the statement is true for $k$ and show that it is true for $k+1$.

By the above claim $V(T_{k+1}) \cap U_k$ consists of only one vertex $u$ and this vertex is strongly associated to hyperedge(s) of $T_s$ for exactly one $1 \le s \le k$. Also, it is easy to see that if $uab \in H_{k+1}$ and $a,b \in V(H_k)$ then $a,b \in V(T_i)$ for some $i \le k$. If $i = s$ and $a,b \in V(T_s)$, then we know by Lemma \ref{skeleton_coloring} that there exists a color for $u$, say $c$ such that hyperedges $uab$ are properly $2$-colored. Let us color $u$ by $c$. If $i \not = s$, and $a,b \in V(T_i)$ then regardless of the color of $u$ the hyperedges $uab$ are $2$-colored properly due to Lemma \ref{skeleton_coloring}.  Since the set of vertices that are strongly associated to hyperedges of $T_{k+1}$ is disjoint from $V(H_k)$ (the already $2$-colored part), we can apply Lemma \ref{skeleton_coloring} to color $T_{k+1}$ such that $u$ is still $2$-colored with $c$ by Observation \ref{observation}. Therefore, we have shown that $H_{k+1}$ is properly $2$-colored such that each $T_i$, $i \le k+1$ is $2$-colored according to Lemma \ref{skeleton_coloring}, as desired and so we have statement for $H_m$ by induction.

In the remaining vertices, namely $V(H) \setminus V(H_m)$, since there are no strongly associated vertices, by Lemma \ref{skeleton_coloring} they can be $2$-colored independently from $H_m$ and we now have a smaller vertex set: $V(H) \setminus V(H_m)$ to color. Therefore, by induction on number of vertices we may $2$-color $H$ properly.
\end{proof}

\subsection{Proof of Lemma \ref{skeleton_coloring} (Main Lemma)}
\label{mainlemma}


We identify some sets of vertices of size $5$ which play an important role in the forthcoming proof.

\begin{definition}
Let $h_1 =abc$, $h_2 =bde$ where $h_1, h_2 \in E(T)$. If there is no hyperedge $h \in H$ such that $\abs{h \cap (h_1 \cup h_2)} = 2$, then the set of vertices $\{a,b,c,d,e\}$ is called a \emph{special block} of $T$.
\end{definition}


\begin{claim}
	\label{matching}
Let $h_1 =abc$, $h_2 =bde$ be thick hyperedges of $T$. If $abe, cbd \in E(H)$ or $abd, cbe \in E(H)$, then $\{a,b,c,d,e\}$ is a special block.
\end{claim}

 \begin{proof}[Proof of Claim \ref{matching}]
Suppose $xyz \in E(H)$ such that $\{x,y,z\} \cap \{a,b,c,d,e\} = \{x,y\}$. It is easy to see that if $x,y\in \{a,c,d,e\}$ then $xyz$ forms a linear triangle with either $h_1,h_2$ or with $abe, cbd$ or with $abd, cbe$. So the only cases that are left to be considered are $\{x,y\} = \{d,b\}$ or $\{x,y\} = \{e,b\}$. Since $\{d,e\}$ is a thick pair either $dea$ or $dec$ is a hyperedge in $H$.  W.l.o.g. let us say $dec \in E(H)$. Then in either of the two remaining cases, $xyz$ along with $abc$ and $dec$ will create a linear cycle, a contradiction.
 \end{proof}

 \begin{claim}
 	\label{two_arrows}
 Let $h_1, h_2$ be thick hyperedges of $T$. If there are two vertices of $h_2$ which are strongly associated to $h_1$, then $h_1 \cup h_2$ is a special block.
 \end{claim}

\begin{proof} [Proof of Claim \ref{two_arrows}]
We will show that $\abs{h_1 \cap h_2} = 1$. Assume by contradiction that $\abs{h_1 \cap h_2} = \emptyset$ and $u, v \in h_2$ are strongly associated to $h_1$. Then it is easy to see that we can choose appropriate supporting hyperedges $h_3, h_4$ of $u$ and $v$, respectively, in $h_2$ such that the hyperedges $h_2, h_3, h_4$ form a linear triangle, a contradiction. 

Let $h_1 = abc$ and $h_2 = bde$, i.e., $d$ and $e$ are strongly associated to $h_1$. Assume by contradiction that there exists a hyperedge $xyz \in H$ such that $\{x,y\} \subset \{a,b,c,d,e\}$ and $z \not \in \{a,b,c,d,e\}$. First let us observe that $\{x,y\} \not \subset \{a,b,c\}$ because the hyperedge $abc$ already has two vertices $d, e$ strongly associated to it and hence cannot have any other vertex associated to it due to Lemma \ref{association}. So if we consider the bipartite graph whose color classes are $\{d,e\}$ and $\{\{a,b\}, \{b,c\}\}$ where $v \in \{d,e\}$ is connected to $\{x,y\} \in \{\{a,b\}, \{b,c\}\}$ if $vxy \in E(H)$, We claim that Hall's condition holds for this bipartite graph. Since the hyperedge $abc$ is thick, using Lemma \ref{association}, $s_{abc}(d) \cup s_{abc}(e) = \{\{a,b\}, \{b,c\}, \{a,c\}\}$. So the union of the neighborhood of $d$ and $e$ in this bipartite graph is $\{\{a,b\}, \{b,c\}\}$. Since $d$ and $e$ are strongly associated to $abc$, they each have at least one neighbor in $\{\{a,b\}, \{b,c\}\}$. So there is a matching by Hall's theorem. 

So either $abe, cbd \in E(H)$ or $abd, cbe \in E(H)$. Now, by applying Claim \ref{matching}, we can conclude that $\{a,b,c,d,e\}$ is a special block.
 \end{proof}

%

 Since the hypergraph induced on $\{a,b,c,d,e\}$ is not $K_5^3$, it is easy to see that there is a proper coloring $\gamma: \{a,b,c,d,e\} \mapsto \{1,2\}$.


\begin{claim}
\label{removeblock}
Assume that $h_1 =abc$, $h_2 =bde$ and $\{a,b,c,d,e\}$ is a special block of $T$. Let $T_a, T_b, T_c, T_d, T_e$ be maximal linear subtrees of $T$ such that $V(T_x) \cap \{a,b,c,d,e\} = \{x\}$ where $x \in \{a,b,c,d,e\}$. Then, if Lemma \ref{skeleton_coloring} holds for each $T_x$, where $x \in \{a,b,c,d,e\}$ and coloring $\gamma: \{a,b,c,d,e\} \mapsto \{1,2\}$ is given, then it holds for $T$ as well.
\end{claim}

\begin{observation}
It is easy to see that $V(T_x) \cap V(T_y) = \emptyset$ for any distinct $x, y \in \{a,b,c,d,e\}$ and $\cup_{x \in \{a,b,c,d,e\} } E(T_x) \cup \{h_1, h_2\} = E(T)$.
\end{observation}

\begin{proof}[Proof of Claim \ref{removeblock}]
Take the $2$-colorings of $T_x$'s ($x \in \{a,b,c,d,e\}$) guaranteed by Lemma \ref{skeleton_coloring} and Observation \ref{observation} such that the color of $x$ is $\gamma(x)$. 

First we show that the hypergraph induced on $V(T)$ is properly $2$-colored. Clearly there is no hyperedge with its vertices in three different $T_x$'s unless it is contained in $\{a,b,c,d,e\}$ because $\{a,b,c,d,e\}$ is a special block and there is no linear cycle in $H$.  

Now we will prove that $w \in V(T_y)$ is not strongly associated to any hyperedge of $T_x$ (for any $y \not = x$). Suppose $w$ is strongly associated to a hyperedge $h$ of $T_x$. Since $w$ is in $T$, by Claim \ref{ass}, there is a hyperedge $h'$ of $T$ which contains $w$ such that $\abs{h \cap h'} = 1$. So $h'$ is a hyperedge of $T$ that has a common vertex with both $T_x$ and $T_y$. Therefore, $h'$ must be either $h_1$ or $h_2$. Moreover, $w = y$ and $h \cap h' =\{x\}$ must hold. Let $h = xpq$. Since $w = y$ is strongly associated to $h$, either $xyp \in E(H)$ or $xyq \in E(H)$, a contradiction to the assumption that $x$ and $y$ belong to a special block. 

So by applying Lemma \ref{skeleton_coloring} to $T_x$, for each hyperedge $uvw$ with $ u, v \in V(T_x)$ and $w \in V(T) \setminus V(T_x)$ the color of $u$ and the color of $v$ are different and so $uvw$ is properly $2$-colored. Since there is no hyperedge with its vertices in three different $T_x$'s, the hypergraph induced by $V(T)$ is properly $2$-colored. 

Let $v \in V(H) \setminus V(T)$. First assume that $v$ is not strongly associated to any hyperedge of $T$ and let $p,q \in V(T)$ be arbitrary. We have to show that if $vpq \in E(H)$ then the colors of $p$ and $q$ are different. If $p,q \in T_x$ for some $x \in \{a,b,c,d,e\}$ then we are done because we assumed Lemma \ref{skeleton_coloring} holds for $T_x$.  So, let $p \in T_x$ and $q \in T_y$ for some distinct $x, y \in \{a,b,c,d,e\}$. Since both $p$ and $q$ can't be in $\{a,b,c,d,e\}$ (by definition of special block), the linear path between $p$ and $q$ in $T$ has at least $2$ hyperedges. This linear path, together with $vpq$ forms a linear cycle, a contradiction.

Now assume that $v$ is strongly associated to a hyperedge of $T$. If $v$ is strongly associated to hyperedges $h_x, h_y$ of $T$ such that $h_x \in E(T_x)$ and $h_y \in E(T_y)$, then as before we can extend a linear path in $T$ between $h_x$ and $h_y$ to a linear cycle by adding appropriate supporting hyperedges of $v$ in $h_x$ and $h_y$. This implies that there is a unique $x \in \{a,b,c,d,e\}$ such that $v$ is strongly associated to hyperedge(s) of $T_x$. Now we show that $v$ can be colored so that all the hyperedges $vpq$ are properly $2$-colored.

By the argument in the previous paragraph if $vpq \in E(H)$ then both $p$ and $q$ are in $T_y$ for some $y \in \{a,b,c,d,e\}$. If $y = x$, then by applying Lemma \ref{skeleton_coloring} to $T_y$, there is a coloring of $v$ such that hyperedges $vpq$ are properly $2$-colored. If $y \not = x$, then $v$ is not strongly associated to any hyperedge of $T_y$. So by applying Lemma \ref{skeleton_coloring} to $T_y$ again, the colors of $p$ and $q$ are different. Therefore, the hyperedges $vpq$ are properly $2$-colored as desired.
\qedhere
\end{proof}

So applying Claim \ref{removeblock} recursively, it suffices to prove Lemma \ref{skeleton_coloring} for a linear subtree $T$ of $H$ which has no special block. So from now on, we may assume that there is no special block in $T$.

We will now construct an auxillary (simple) graph $G_T$ by following the steps in the \emph{Construction} below, one after another. This graph is connected, and its vertex set and edge set satisfy: $V(G_T) = V(T)$ and if $ab \in E(G_T)$ then there exists a vertex $x \in V(T)$ such that $abx \in E(T)$. We show later that this graph $G_T$ is actually a tree and that a proper $2$-coloring of $G_T$ will give us a proper $2$-coloring of the hypergraph induced on $V(T)$ as demanded by Lemma \ref{skeleton_coloring}.

\begin{cons}
We perform the steps as follows. First Step 1 as long as we can, then Step 2 as long as we can, and so on. Naturally, edges added earlier will not be added again. 

\begin{itemize}
\item [Step 1.] If $abc, ebd \in E(T)$, $abc$ is a thick hyperedge and $e$ is strongly associated to $abc$ then,
\begin{enumerate}
  \item [(a)] add $eb$ to $E(G_T)$.
  \item [(b)] and if $ace \in E(H)$ also holds, then add $ac$ to $E(G_T)$ as well.
\end{enumerate}

\item [Step 2.] If $abc \in E(T)$, $vab \in E(H)$ and $v$ is weakly associated to $abc$, then add $ab$ to $E(G_T)$.

\item [Step 3.] If $abc, ebd \in E(T)$, $v \in V(H)\setminus V(T)$ is strongly associated to $abc$ and $ebd$, and if $acv$ (respectively $edv$) is a hyperedge of $H$, then add $ac$ (respectively $ed$) to $E(G_T)$.

\item [Step 4.]After completing the above steps, for every hyperedge $abc \in E(T)$ we do the following. If $abc$ is thick, and less than two of the three pairs $ab, bc, ca$ are in $E(G_T)$ we add some more pairs arbitrarily so that $E(G_T)$ has exactly two pairs from $ab, bc, ca$. If $abc$ is not thick and less than two of the three pairs $ab, bc, ca$ are in $E(G_T)$, we add pairs from $ab, bc, ca$ such that only one pair remains outside $E(G_T)$ and it is not a thick pair.
\end{itemize}
\end{cons}

\begin{remark}
Notice that all edges $xy$ added in Steps 1, 2, 3 satisfy that $\{x, y\}$ is a thick pair. 
\end{remark}

Now we claim the following.
\begin{claim}
\label{tree_lemma}
$G_T$ is a tree (so it can be properly $2$-colored).
\end{claim}

Before we prove the above claim, we will show that it implies Lemma \ref{skeleton_coloring}.

First let us prove that a proper $2$-coloring of $G_T$ gives us a proper $2$-coloring of the subhypergraph induced by $V(T)$. Since $V(G_T) = V(T)$, a proper $2$-coloring of $G_T$ gives us a proper $2$-coloring of the hyperedges of $T$. Therefore, it suffices to prove that for every hyperedge $abc \in E(T)$, the hyperedges $xyv$ where $x,y \in \{a,b,c\}$ and $v \in V(T) \setminus \{a,b,c\}$ are properly $2$-colored. If $abc$ is not thick, then it is easy to see that $xy$ (which has to be a thick pair) must be in $G_T$ (due to Step 4 of Construction of $G_T$) which means that $x$ and $y$ have different colors and so the hyperedge $xyv$ is properly $2$-colored, as desired. If $abc$ is thick, then $v$ must be associated to $abc$. If $v$ is weakly associated to $abc$, then by the construction of $G_T$ (Step 2 of Construction), $xy$ must be in $G_T$ and so $xyv$ is properly $2$-colored again. If $v$ is strongly associated to $abc$, then by Claim \ref{ass}, $v$ belongs to a hyperedge $h$ neighboring $abc$ in $T$ (i.e., $\abs{h \cap abc} = 1$). W.l.o.g we may assume that  $h \cap abc = \{b\}$, and let $h := vbw$. By Construction Step 1a and 1b of $G_T$, we have $bv, ac \in E(G_T)$ if $acv \in E(H)$. So $b$ and $v$ have different colors and $a$ and $c$ have different colors. Therefore, all the hyperedges $vxy$ are properly $2$-colored. So the subhypergraph induced by $V(T)$ is properly $2$-colored.

Now let $v \in V(H) \setminus V(T)$. Note that for any $xyv \in E(H)$ where $x, y \in V(T)$, $x, y$ must belong to a hyperedge of $T$. We will show that $v$ can be colored as required in Lemma \ref{skeleton_coloring}. If $v$ is not strongly associated to any hyperedge of $T$, then for every $xyv \in E(H)$, $xy \in E(G_T)$ and so $xyv$ is properly $2$-colored regardless of the color of $v$. So assume that $v$ is strongly associated to hyperedges $h_1, h_2, \ldots, h_k$ of $T$. We consider two cases. If $k \ge 2$, then by Claim \ref{ass}, $\abs{h_i \cap h_j} \not = \emptyset$ for every $i, j \in \{1,2,\ldots,k\}$. Since $h_i$ are hyperedges of a linear tree, and every two of them have a common vertex, there is a vertex $o$ such that $\cap_i h_i = \{o\}$. Let us show that choosing the color of $v$ to be different from the color of $o$ guarentees that all the hyperedges $xyv \in E(H)$ for $x,y \in V(T)$ are properly $2$-colored, as required by Lemma \ref{skeleton_coloring}. If $\{x, y \} \not \subseteq h_i$ for any $i$, then as we saw before $xyv$ is properly $2$-colored independent of the color of $v$. So $xy \in h_i$ for some $i$. If $o \in \{x,y\}$, then since $o$ and $v$ are colored differently, $xyv$ is $2$-colored properly. If $o \not \in \{x,y\}$, then by the construction of $G_T$ (see Construction Step $3$), $xy$ is in $G_T$ and so $xyv$ is properly $2$-colored, as desired. So the only remaining case is if $k = 1$. In this case, the hyperedge $h_1$ has two vertices of the same color and if we color $v$ differently from this color, hyperedges $vxy$ where $x,y \in V(T)$ are properly $2$-colored. This completes the proof of Lemma \ref{skeleton_coloring}.

\begin{proof} [Proof of Claim \ref{tree_lemma}]
Notice that $G_T$ is connected as guaranteed by Construction Step 4. Assume by contradiction that $G_T$ has a cycle. Since $T$ is a linear tree, this cycle has to be a triangle $abc$ where $abc \in E(T)$ is a thick hyperedge. First observe that none of the pairs $ab,bc,ca$ were added during Step $4$ of the construction of $G_T$.  We now consider different cases for how $abc$ could be formed.

\begin{case}
One of the pairs $ab, bc, ca$ was added by Construction Step 1b.
\end{case}

W.l.o.g let the pair added by Construction Step 1b was $ac$. Then, there exists a hyperedge $ebd \in E(T)$ such that $e$ is strongly associated to $abc$ and $ace \in E(H)$. So either $abe$ or $bce$ is in $E(H)$. Clearly, there is no $w \not \in \{a,b,c,d,e\}$ such that $wab$ or $wbc$ is a hyperedge of $H$ for otherwise we have a linear cycle. Since $abc$ is thick, $ab, bc$ are thick pairs.
If either $bcd$ or $abd$ is in $E(H)$, then the conditions of Hall's theorem hold for the bipartite graph whose color classes are $\{ab, bc\}$ and $\{d,e\}$ where $xy \in \{ab, bc\}$ is connected to $z \in \{d,e\}$ if and only if $xyz \in E(H)$. So there is a matching and by Claim \ref{matching}, we have a contradiction since we assumed there is no special block of $T$.
So assume that $bcd, abd \not \in E(H)$. So the only hyperedges (besides $abc$) containing $ab$ and $bc$ are $abe$ and $bce$ which implies that $ab$ and $bc$ were not added by Construction Steps $1b$, $2$ and $3$. So both $ab$ and $bc$ were added by Construction Step $1a$. Assume that $bc$ was added because either $b$ or $c$ was strongly associated to a hyperedge $h'$. This means that $h'$ is thick and $h' = bde$ because otherwise we have $wbc \in E(H)$ for some $w \not \in \{a,b,c,d,e\}$, a contradiction. So $c$ is strongly associated to $bde$. Similarly, $a$ is strongly associated to $bde$. So by Claim \ref{two_arrows}, $\{a,b,c,d,e\}$ is a special block, a contradiction.

So from now on, we can assume that Construction Step 1b was never used to add the pairs $ab, bc, ca$.

\begin{case}
	One of the pairs $ab, bc, ca$ was added by Construction Step $3$.
\end{case}
W.l.o.g let us say $ac$ was added by Construction Step $3$. Then, there is a hyperedge $bde \in E(T)$ and $v \in V(H) \setminus V(T)$ such that $v$ is strongly associated to both hyperedges $abc$, $bed$ and $acv \in E(H)$.
Since $ab$ is a thick-pair, there is a vertex $w \not \in \{a,b,c\}$ such that $abw \in E(H)$. If $w \not \in \{a,b,c,d,e,v \}$ then since $acv, wab \in E(H)$ and one of $bev, bdv \in E(H)$, they form a linear cycle, a contradiction. If $w = e$, then since $abe, acv \in E(H)$ and one of $bdv, dev \in E(H)$, we have a linear cycle again, a contradiction. Similarly $w \not= d$. Therefore, $w = v$. So the only hyperedge besides $abc$ which contains $ab$, is $abv$. Similarly, the only hyperedge besides $abc$ which contains $bc$ is $bcv$. This implies that $ab$ and $bc$ were not added by Construction Step 1, 2 and 4. Also, it's easy to see that they were not added by Construction Step 3, otherwise $v$ would have been strongly associated to a hyperedge of $T$ which is not a neighbor of $ebd$, which is a contradiction.

So the only reminaing case is when $ab, bc, ca$ are added by Construction Step $1a$ or $2$.

\begin{case}
$ab, bc, ca$ were added by Construction Step $1a$ or $2$.
\end{case}
Two of the pairs $ab, bc, ca$ cannot be added by Construction Step $2$ due to Lemma \ref{association}. Therefore, we have two subcases: Either exactly one of $ab, bc, ca$ was added by Construction Step $2$ and the other two were added by Construction Step $1a$ or all of them were added by Construction Step $1a$.

Assume that all of the pairs $ab, bc, ca$ were added by Construction Step $1a$. Let $xy \in \{ab, bc, ca\}$. Let us say $xy$ was added because there is a thick hyperedge $h_{xy} \in E(T)$ which is strongly associated to either $x$ or $y$. If any two of the there hyperedges $h_{ab}, h_{bc}, h_{ca}$ are the same, then by Claim \ref{two_arrows}, we have a special block in $T$, a contradiction. Therefore, $h_{ab} \not= h_{bc} \not= h_{ca}$. But then, we have hyperedges $abv_1, acv_2, bcv_3 \in E(H)$ where $v_1 \in h_{ab}, v_2 \in h_{bc}, v_3 \in h_{ac}$ which form a linear cycle, a contradiction.

Now assume that one of the pairs $ab, bc, ca$ was added by Construction Step $2$ and the other two were added by Construction Step $1a$. W.l.o.g assume that $ab$ and $bc$ were added by Construction Step $1a$ and $ca$ by Construction Step $2$.  Let us say $ab$ (respectively $bc$) was added because there is a thick hyperedge $h_{ab} \in E(T)$ (respectively $h_{bc} \in E(T)$) which is strongly associated to either $a$ or $b$ (respectively $b$ or $c$). So there are vertices $v_1 \in h_{ab}$ and $v_2 \in h_{bc}$ such that $abv_1, bcv_2 \in E(H)$. If $h_{ab} = h_{bc}$, then by Claim \ref{two_arrows} we have a special block in $T$, a contradiction. So $h_{ab} \not= h_{bc}$ as before. Let us say $ac$ was added because there is a vertex $w$ weakly associated to $abc$ such that $wac \in E(H)$. If $w \not = v_1$ and $w \not = v_2$, then we have a linear cycle, namely $acw, abv_1, bcv_2$, a contradiction. So let us assume w.l.o.g that $w = v_1$. Let $h_{ab} = v_1ex$ where $x$ is either $a$ or $b$. If $x = b$, then $h_{ab}, v_1ac, bcv_2$ is a linear cycle, a contradiction. If $x = a$, then clearly $b$ is strongly associated to $h_{ab} = v_1xe$. So either the hyperedge $abe \in E(H)$ or $bev_1 \in E(H)$. This hyperedge together with $acv_1$ and $bcv_2$ gives us a linear cycle, a contradiction.
\end{proof}

\section{Proof of Theorem \ref{degree_condition}: A degree condition for linear-cycle-free hypergraphs}
\label{degree}
Let $H$ be a $3$-uniform hypergraph without any linear cycles. The following is our main lemma. 

\begin{lemma}
\label{no_universalpair}
If there are no vertices $u, v \in V(H)$ such that $uvx \in E(H)$ for all $x \in V(H) \setminus \{u,v\}$ then there is a vertex of degree at most $\abs{V(H)}-2$ whenever $\abs{V(H)} \ge 6$.
\end{lemma}

\vspace{2mm}

We prove Lemma \ref{no_universalpair} in Section \ref{no_universalpair_proof}. Using this lemma, we will prove Theorem \ref{degree_condition} in Section \ref{universal_pair}.

\subsection{Proof of Lemma \ref{no_universalpair}}
\label{no_universalpair_proof}
First let us prove some preliminary lemmas.

\subsubsection*{Preparatory lemmas}
The length of a linear path is defined as the number of hyperedges in it. Let $k$ be the length of a longest linear path in $H$. Among all skeletons that contain a linear path of length $k$, let $T$ be a skeleton of maximum possible size. Below we prove some lemmas about such a skeleton.

\begin{lemma}
	\label{onestrong_allweak}
	Any hyperedge $abc \in E(T)$ is strongly associated to at most one vertex of $V(H) \setminus V(T)$.
\end{lemma}

\begin{proof}
	Suppose by contradiction that $abc \in E(T)$ is strongly associated to two vertices $v_1, v_2 \in V(H) \setminus V(T)$. Consider the bipartite graph whose color classes are $\{v_1, v_2\}$ and $\{ab, bc, ca\}$ where $v \in \{v_1, v_2\}$ and $xy \in \{ab, bc, ca\}$ are adjacent iff $vxy \in E(H)$. Then it can be easily seen that there is a matching saturating $\{v_1, v_2\}$ between the two color classes. If we replace $abc$ by the two hyperedges corresponding to this matching we will get a skeleton of bigger size and it is easy to see that the length of the longest linear path in it does not decrease, a contradiction.
\end{proof}

In the remainder of this paper the degree of a vertex $v \in V(T)$ in the subhypergraph of $H$ induced by $V(T)$ is denoted by $d_{T}(v)$. 
We have the following corollary of the above lemma.
\begin{corollary}
	\label{outside_degree}
	Let $\abs{V(H) \setminus V(T)} = t$. Then the degree of any vertex $v \in V(T)$ which is in exactly one hyperedge of $T$, is at most $d_{T}(v) + t+1$.
\end{corollary}
\begin{proof}
Let $uvw$ be the hyperedge of $T$ containing $v$. The total number of hyperedges incident on $v$ is $d_{T}(v)$ plus the number of hyperedges incident on $v$ that contain a vertex from $V(H) \setminus V(T)$.

It is easy to check that if $x \in V(H) \setminus V(T)$, then at most two hyperedges of $H$ contain both $v$ and $x$: namely $vxu$ and $vxw$. Moreover, if both of them are in $H$ then $x$ is strongly associated to $uvw$ and there is at most one such $x$ by Lemma \ref{onestrong_allweak}. Therefore, for all $x \in V(H) \setminus V(T)$ except at most one, there is at most one hyperedge containing $v$ and $x$. Thus the corollary follows.
\end{proof}

\begin{definition}[star]
	\label{star_opposite_pair}
	\emph{Star} of the skeleton $T$ at $v \in V (T)$ is defined as the subtree of $T$ consisting of the hyperedges of $T$ incident to $v$. The vertex $v$ is called the \emph{center} of this star.
\end{definition}

\begin{definition}[opposite pair]
	\label{opposite_pair_def}
	Let us define a graph $G(T )$ consisting of all the pairs covered by the hyperedges of the skeleton $T$. For a vertex $v \in V (T )$ and a vertex pair $\{x, y\}$ such that $xy \in E(G(T))$, we say $\{x, y\}$  is \emph{opposite} to $v$ if $x$ and $y$ are at equal distance from $v$ in $G(T)$. This equal distance is also called the distance between $v$ and the opposite pair $\{x, y\}$.
	
	Note that every hyperedge of $T$ contains exactly one pair opposite to $v$.  
	
\end{definition}

In the next lemma, by means of opposite pairs, we can describe all the hyperedges intersecting a given star exactly in its center.

\begin{lemma}
	\label{opposite_pair}
	Let $v \in V(T)$ and $vab \in E(H)$ be such that $a, b$ are not contained in the star at $v \in  V (T )$. Then $\{a, b\}$ is a pair opposite to $v$ in $T$.
\end{lemma}
\begin{proof}
	Since $T$ is a skeleton of maximum size (among those skeletons containing a linear path of length $k$), clearly it is impossible that $a, b \in V(H) \setminus V(T)$. Moreover, if exactly one of $a, b$ is in $V(T)$, then since $\{a, b\}$ does not intersect the star at $v \in  V (T )$, it is easy to find a linear cycle, a contradiction. Therefore, both $a, b$ are in $V(T)$. Now assume for a contradiction that $\{a, b\}$ is a pair which is not opposite to $v$ in $T$. Without loss of generality let us assume that distance from $v$ to $b$ in $G(T)$ is strictly smaller than the distance from $v$ to $a$. Then it is easy to see that the linear path between $v$ and $b$ in $T$ does not contain $a$, so the hyperedge $vba$ together with this linear path forms a linear cycle, a contradiction.
\end{proof}


\begin{lemma}
	\label{prison}
	Let $\{p_0q_0p_1, p_1q_1p_2, p_2q_2p_3, \ldots, p_{k-2}q_{k-2}p_{k-1}, p_{k-1}q_{k-1}p_k\}$ be a linear path in $T$. Let $p_0q_0x \in E(H)$ for some $x \in V(T)$ and let us consider the linear path between $x$ and $p_0$. Let $\P'$ be the subpath of this linear path without the starting and ending hyperedges (i.e., not including the two hyperedges which contain $p_0$ and $x$). Then, for any $y,z\in V(\P') \setminus \{p_1\}$, $p_0yz \not \in E(H)$.
\end{lemma}
\begin{proof}
Suppose for a contradiction that $p_0yz \in E(H)$ for some $y,z\in V(\P') \setminus \{p_1\}$. Since $yz$ does not intersect the star at $p_0$, by Lemma \ref{opposite_pair}, $yz$ is a pair opposite to $p_0$ in $T$. Then it is easy to see that $p_0yz$, $p_0q_0x$ and the linear path between the pair $\{y,z\}$ and $x$ in $T$ form a linear cycle, a contradiction.
\end{proof}

We are now ready to prove Lemma \ref{no_universalpair}. We divide its proof into two cases depending on whether the length of a longest linear path in $H$ is at least 3 or at most 2 (in Section \ref{kge3} and Section \ref{kleq2} respectively).

\subsubsection{The length of a longest linear path in $H$ is at least 3}
\label{kge3}

 Let $k$ be the length of the longest linear path in $H$. 

\begin{definition}[windmills]
Given a linear path $\{p_0q_0p_1, p_1q_1p_2, p_2q_2p_3, \ldots, p_{k-2}q_{k-2}p_{k-1}, p_{k-1}q_{k-1}p_k\}$ of length $k$ in $H$ and a skeleton containing it, the set of hyperedges of this skeleton which contain $p_1$ (respectively $p_{k-1}$) except $p_1q_1p_2$ (respectively except $p_{k-2}q_{k-2}p_{k-1}$) is called as a \emph{windmill} at $p_1$ (respectively $p_{k-1}$) and the size of this set is called the size of the windmill. In other words, windmill at $p_1$ is a star at $p_1$ minus the hyperedge $p_1q_1p_2$ (and similarly, windmill at $p_{k-1}$ is a star at $p_{k-1}$ minus the hyperedge $p_{k-2}q_{k-2}p_{k-1}$).

So there are two windmills corresponding to a linear path of length $k$ and a skeleton containing it. The windmill of smaller size among the two is referred as the smaller windmill. If they are of same size, then either one can be considered as the smaller windmill. 

Note that as we assumed $k \ge 3$, the two windmills do not have any hyperedges in common.
\end{definition}

Among all skeletons that contain a linear path of length $k$, let us consider skeletons that are of maximum possible size (so preparatory lemmas of the previous section can still be applied). Now among these skeletons let us choose a skeleton $T$ and a linear path $\P$ of length $k$ in $T$ such that the size of the smaller windmill corresponding to $T$ and $\P$ is minimum.  Let $\P = \{p_0q_0p_1, p_1q_1p_2, p_2q_2p_3, \ldots, p_{k-1}q_{k-1}p_k\}$. Without loss of generality, we may assume that the smaller windmill is at $p_1$.

\vspace{2mm}

We distinguish two cases depending on the size of the smaller windmill corresponding to $T$ and $\P$.

\begin{namedthm*}{Case 1}
The size of the smaller windmill (corresponding to $T$ and $\P$) is at least $2$.
\end{namedthm*}

We will show that the degree of $p_0$ is at most $\abs{V(H)}-2 = n-2$.

 If $x$ is in $V(T) \setminus \{p_1, p_0, q_0\}$, then we claim that $p_0q_0x \not \in E(H)$ because if $x$ is in the windmill around $p_1$ then the linear path $\P$ can be extended. If $x$ is not in the windmill around $p_1$ then by replacing the hyperedge $p_0q_0p_1$ with $p_0q_0x$ will decrease the size of the smaller windmill (and the length of the longest linear path in the skeleton does not decrease) contradicting the assumption that the size of the smaller windmill is minimum.

So the hyperedges in $V(T)$ containing $p_0$ are of the type $p_0p_1x$ where $x \in V(T) \setminus \{q_0\}$ or of the type $p_0xy$ where $x, y \in V(T) \setminus \{p_1, q_0\}$ plus the hyperedge $p_0q_0p_1$. Below we will count the number of hyperedges of these two types separately.

First, let us count the number of hyperedges of the type $p_0p_1x$ where $x \in V(T) \setminus \{q_0\}$. Since $p_0p_1$ can't be opposite to any $x \in V(T) \setminus \{q_0\}$, by Lemma \ref{opposite_pair}, $p_0p_1$ must intersect the star at $x$. This means that $x$ should be contained in the star at $p_1$. So the number of hyperedges of the type $p_0p_1x$ where $x \in V(T) \setminus \{q_0\}$ is $2w_1$ where $w_1$ is the size of the windmill at $p_1$ (note that by definition, windmill at $p_1$ does not contain the edge $p_1q_1p_2$). Let $w_2$ be the size of the windmill at $p_{k-1}$ (So $w_1 \le w_2$ by our assumption).

Now, let us count the number of hyperedges of the type $p_0xy$ where $x, y \in V(T) \setminus \{p_1, q_0\}$. Since $xy$ doesn't intersect the star at $p_0$, by Lemma \ref{opposite_pair}, $xy$ must be opposite to $p_0$. If the pair $xy$ is contained in a hyperedge of either windmill (at $p_1$ or $p_{k-1}$) then we can extend $\P$ by $p_0xy$, a contradiction. 

 So the number of such $xy$ pairs is at most $$\frac{V(T) - (2w_1+1)- 2w_2}{2} = \frac{(n-t)- (2w_1+1)- 2w_2}{2}.$$

Therefore, the total degree of $p_0$ in the subhypergraph induced by $V(T)$, $$d_T(p_0) \le 1 + 2w_1 + \frac{(n-t)- (2w_1+1)- 2w_2}{2}.$$ Thus by Corollary \ref{outside_degree}, the degree of $p_0$ is at most $$1 + 2w_1 + \frac{(n-t)- (2w_1+1)- 2w_2}{2} + t +1 =  \frac{n +t + 2w_1- 2w_2 + 3}{2} \le \frac{n +t + 3}{2}.$$
So we are done unless $ \frac{n +t + 3}{2} \ge n-1$, which simplifies to $n-t = \abs{V(T)} \le 5$, so $T$ contains at most 2 hyperedges. Therefore the length of $\P$ is at most 2 (recall that $\P$ is contained in $T$). However, as $\P$ is a longest linear path in $H$,  this contradicts the assumption of Section \ref{kge3}.

\begin{namedthm*}{Case 2}
	The size of the smaller windmill (corresponding to $T$ and $\P$) is $1$.
\end{namedthm*}

There are three types of hyperedges in $H$ that contain $p_0$: hyperedges of the type $p_0q_0x$ where $x \in V(H) \setminus \{p_0, q_0\}$, hyperedges of the type $p_0yz$ or of the type $p_0p_1w$ where $y, z, w \in V(H) \setminus \{p_0, q_0, p_1\}$. (Note that we consider the hyperedge $p_0q_0p_1$ as of the type $p_0q_0x$.)

Let $r$ be the number of hyperedges in $H$ of the type $p_0q_0x$ where $x \in V(H) \setminus \{p_0, q_0\}$ and let $s$ be the number of hyperedges in $H$ of the type $p_0yz$ where $y,z \in V(H) \setminus \{p_0, q_0, p_1\}$. Below we upper bound the number of hyperedges  of these two types together.

\begin{claim}
	\label{rs}
	$r + s \le n -2$ and if equality holds then $p_0p_kq_{k-1} \in E(H)$.
\end{claim}

\begin{proof}
First we claim that $r + s \le n - s$. Consider a hyperedge of the type $p_0yz$ where $y,z \in V(H) \setminus \{p_0, q_0, p_1\}$. Since $\{y,z\}$ doesn't intersect the star at $p_0$, by Lemma \ref{opposite_pair}, the pair $\{y,z\}$ is opposite to $p_0$. We claim that if $p_0yz \in E(H)$ then the pair $\{y,z\}$ must be contained in the linear path $\P$. It is easy to see that since $\{y,z\}$ is opposite to $p_0$, either both $y$ and $z$ are contained in $\P$ or both of them are not in $\P$. In the latter case, $\P$ can be extended by adding the hyperedge $p_0yz$, contradicting the maximality of $\P$. So $y$ and $z$ are contained in $\P$. Now consider the opposite pair $\{y_1,z_1\}$ closest (in the sense of distance defined in Definition \ref{opposite_pair_def}) to $p_0$ in $\P$ such that $p_0y_1z_1 \in E(H)$. 
By Lemma \ref{prison}, the farthest $x \in \P$ from $p_0$ such that $p_0q_0x \in E(H)$ can be either $y_1$ or $z_1$. This means that a vertex in $V(H) \setminus \{p_0, q_0\}$ can not be contained in both a hyperedge of type $p_0q_0x$ and a hyperedge of type $p_0yz$ except for the vertices $y_1, z_1$. Since the hyperedges of the type $p_0yz$ cover $2s$ vertices from $V(H) \setminus \{p_0, q_0\}$ and hyperedges of the type $p_0q_0x$ cover $r$ vertices from $V(H) \setminus \{p_0, q_0\}$, we have $r + 2s \le n-2 + 2 = n$, proving that $r + s \le n - s$.

Since $r + s \le n - s$, Claim \ref{rs} is proved if $s \ge 2$ and so we can assume $s \le 1$. Recalling the assumption of Lemma \ref{no_universalpair}, there are no vertices $u, v \in V(H)$ such that $uvx \in E(H)$ for every $x \in V(H) \setminus \{u,v\}$, so we have $r \le n-3$. Thus, $$r + s \le n -3 +1 = n-2,$$ as desired.

Now let us observe what happens when $r+s = n-2$. Then we must have $s \ge 1$ (as $r \le n-3$). That is, there exists a hyperedge of the type $p_0yz$ where $y,z \in V(H) \setminus \{p_0, q_0, p_1\}$. The pair $\{y,z\}$ must be opposite to $p_0$ and is contained in $\P$ as before. So if $\{y,z\} \not = \{p_k,q_{k-1}\}$ then by Lemma \ref{prison}, $p_0q_0p_k, p_0q_0q_{k-1} \not \in E(H)$ (here we used the existence of an edge of the type $p_0yz$ and that $p_k$ and $q_{k-1}$ are further than $y, z$). So the vertices $p_k, q_{k-1}$ do not belong to any hyperedge of type $p_0q_0x$ or $p_0yz$. So, by a similar argument as in the previous paragraph, $r + 2s \le n-4 + 2 = n-2$ which is a contradiction since we assumed $r + s = n-2$ and $s \ge 1$.
\end{proof}

We distinguish two subcases based on the existence of a hyperedge of certain type.

\begin{namedthm*}{Case 2.1}
There is a hyperedge of type $p_0q_0x \in E(H)$ where $x \in V(T) \setminus \{ p_0,p_1,p_2,q_0,q_1\}$.
\end{namedthm*}

In this case, we claim that number of hyperedges of the type $p_0p_1y$ in $H$ where $y \in V(H) \setminus \{p_0, q_0, p_1\}$ is at most $1$ and if such a hyperedge exists then $y$ is either $p_2$ or $q_1$. Assume that $p_0p_1y' \in E(H)$ where $y' \in V(H) \setminus \{p_0, q_0, p_1\}$. Let $\P_1$ be a linear path in $T$ between (and including) $x$ and $p_1$. If $y' \not \in \P_1$, then $p_0q_0x$, $p_0p_1y'$ and $\P_1$ form a linear cycle. So $y' \in \P_1$. Since $\{p_0, p_1\}$ cannot be an opposite pair of any vertex on $\P_1$ except $q_0$, by Lemma \ref{opposite_pair}, $\{p_0, p_1\}$ must intersect the star at $y'$. So $y'$ is either $p_2$ or $q_1$. If both hyperedges $p_0p_1p_2$ and $p_0p_1q_1$ are in $H$ then $p_0q_0x$, $\P_1 \setminus \{p_1p_2q_1\}$ and either $p_0p_1p_2$ (in case $p_2$ is on the path $\P_1 \setminus \{p_1p_2q_1\}$) or $p_0p_1q_1$ (in case $q_1$ is on the path $\P_1 \setminus \{p_1p_2q_1\}$) form a linear cycle. Therefore the desired claim follows.

If neither of the hyperedges $p_0p_1p_2$, $p_0 p_1q_1$ are in $H$, then the degree of $p_0$ is $r+s$ and by Claim \ref{rs}, $r+s \le n-2$ and so Lemma \ref{no_universalpair} holds. Therefore, from now on, we may assume that exactly one of the two hyperedges $p_0p_1p_2$, $p_0 p_1q_1$ is in $H$.
If $r + s$ is strictly less than $n-2$ then degree of $p_0$ is at most $n-2$ and Lemma \ref{no_universalpair} holds again. So we also assume that $r+s=n-2$. By Claim \ref{rs} if $r+s=n-2$, then $p_0p_kq_{k-1} \in E(H)$. It follows that the size of the windmill at $p_{k-1}$ is $1$ because if it is more than $1$, then the linear path consisting of $p_0p_kq_{k-1}$, $\P \setminus p_{k-1}p_kq_{k-1}$ and one of the hyperedges of the windmill at $p_{k-1}$ (different from $p_{k-1}p_kq_{k-1}$) form a linear path longer than $\P$, a contradiction. Therefore the size of the windmills at $p_{k-1}$ and $p_1$ are both $1$.  By symmetry, if we define $r'$ and $s'$ for $p_k$ analogous to how we defined $r$ and $s$ for $p_0$, Claim \ref{rs} holds for them. Since a hyperedge of the type $p_kq_{k-1}x$ exists where $x \in V(T) \setminus \{p_k,q_{k-1}, p_{k-1}, q_{k-2}, p_{k-2}\}$ (namely $p_kq_{k-1}p_0$), by repeating the same argument as before we can assume that $r' + s' = n - 2$ and so $p_0q_0p_k \in E(H)$. Using Lemma \ref{prison} for $p_k$ (instead of $p_0$), 
it is easy to see that $s' \le 1$. So $r' \ge n-3$. We know that $p_0p_1y \in E(H)$ where $y$ is either $p_2$ or $q_1$. Now $p_0p_1y$, $p_0q_0p_k$ and either $p_kq_{k-1}y$ or $p_kq_{k-1}p_1$ (one of them exists because $r' \ge n-3$) form a linear cycle, a contradiction.

\begin{namedthm*}{Case 2.2}
There is no hyperedge of type $p_0q_0x \in E(H)$ where $x \in V(T) \setminus \{ p_0,p_1,p_2,q_0,q_1\}$.
\end{namedthm*}

Let $d_0$ be the degree of $p_0$ in the subhypergraph of $H$ induced by $\{ p_0,p_1,p_2,q_0,q_1\}$. Clearly $d_0 \le 6$. If $p_k q_{k-1}p_0$ or $p_k q_{k-1}q_0$ are in $H$, then the size of the windmill at $p_{k-1}$ is $1$, otherwise the linear path consisting of $p_kq_{k-1}p_0$, $\P \setminus p_{k-1}p_kq_{k-1}$ and one of the hyperedges of the windmill at $p_{k-1}$ form a linear path longer than $\P$, a contradiction. So by symmetry (renaming $p_i$ to $p_{k-i}$ for each $0 \le i \le k$ and $q_i$ to $q_{k-1-i}$ for each $0 \le i \le k-1$) we are done by Case 2.1. Thus we can assume 
\begin{equation}
\label{recall}
p_k q_{k-1}p_0, p_k q_{k-1}q_0 \not \in E(H).
\end{equation}

If there is a vertex $v \in V(H) \setminus V(T)$ which is strongly associated to $p_0q_0p_1$, then we claim that $d_0 \le 4$ because if either $p_0q_0p_2$ or $p_0q_0q_1$ is in $H$, then it is easy to check that we have a linear cycle. Let $\abs{V(H) \setminus V(T)} = t$. So the degree of $p_0$ in the subhypergraph of $H$  induced by $T$, $d_T(p_0) \le d_0 + \frac{n-t-7}{2}$ (here we used $p_k q_{k-1}p_0 \not \in E(H)$).  By Corollary \ref{outside_degree}, degree of $p_0$ is at most $$d_0 + \frac{n-t-7}{2} + t + 1 \le \frac{n+t+3}{2}.$$
Then, Lemma \ref{no_universalpair} holds unless $\frac{n+t+3}{2} \ge n-1$ which simplifies to $n-t = \abs{V(T)} \le 5$, so $T$ contains at most 2 hyperedges. Therefore the length of $\P$ - a longest linear path of $H$- is also at most 2 contradicting the assumption of Section \ref{kge3}.

If there is no vertex $v \in V(H) \setminus V(T)$ which is strongly associated to $p_0q_0p_1$, then degree of $p_0$ is at most $d_T(p_0) + t$. And, since $d_T(p_0) \le d_0 + \frac{n-t-7}{2} $, the degree of $p_0$ is at most
$$d_0 + \frac{n-t-7}{2} + t \le d_0 + \frac{n+t-7}{2},$$

and Lemma \ref{no_universalpair} holds unless $d_0 + \frac{n+t-7}{2} \ge n-1$ which simplifies to 
\begin{equation}
\label{eq:d0}
d_0 \ge \frac{n-t+5}{2}.
\end{equation}

 If $n-t > 7$ then $d_0 > 6$ which is impossible. So we may assume $n- t \le 7$. If $n-t = \abs{V(T)} \le 5$ then $T$ contains at most two hyperedges, so the length of $\P$ is also at most 2 contradicting the assumption of Section \ref{kge3}. Since $n-t$ is odd (as the number of vertices in a skeleton is always odd) the only remaining case is when $n-t = 7$. In this case the size of the skeleton $T$ is $3$ and since $T$ contains a linear path of length at least 3 (as we are in Section  \ref{kge3}), $T$ is a linear path of length exactly $3$ (i.e., $T$ and $\P$ contain the same set of hyperedges). Thus $k = 3$. Moreover, by \eqref{eq:d0}, $d_0 \ge 6$. However, since $d_0 \le 6$, we have  $d_0 = 6$. By symmetry, the degree of $q_0$ in the subhypergraph induced by $\{ p_0,p_1,p_2,q_0,q_1\}$ is also $6$.  
 This implies that 
 \begin{equation}
 \label{butterfly}
q_0p_1q_1, p_0p_1p_2, p_0p_2q_1 \in E(H).
 \end{equation}
 
By \eqref{recall}, we can assume $p_3q_2p_0, p_3q_2q_0 \not \in E(H)$. Recall that $p_3p_0q_0 \not \in E(H)$. Therefore, any hyperedge containing $p_3$ in the subhypergraph induced by $V(T)$ is contained in $\{ p_3,q_2,p_2,q_1,p_1\}$. However, by \eqref{butterfly} the two hyperedges $p_3p_2p_1, p_3p_2q_1 \not \in E(H)$, since otherwise we can find a linear cycle in $H$. Therefore, the degree of $p_3$ in the sybhypergraph induced by $V(T)$ is $d_{T}(p_3) \le 6-2 = 4$. Thus, by Corollary \ref{outside_degree}, degree of $p_3$ is at most $d_T(p_3)+\abs{V(H) \setminus V(T)}+1 \le 5 + \abs{V(H) \setminus V(T)} =  \abs{V(H)} - 2$, as desired, finishing the proof of this case.



\subsubsection{The length of a longest linear path in $H$ is at most 2}
\label{kleq2}

Let $k$ be the length of the longest linear path in $H$. So $k \le 2$.
Among all skeletons that contain a linear path of length $k$, let $T$ be a skeleton of maximum possible size.  

As the length of a longest linear path in $H$ is at most $2$, it is easy to see that all of the hyperedges in $T$ share a common vertex, $b$. We consider the following three cases depending on the number of hyperedges in $T$. 

\begin{namedthm*}{Case 1}
	\label{atleast5}
$T$ consists of at least 3 hyperedges.
\end{namedthm*}
Let $E(T) = \{v_1v_2b$, $v_3v_4b, \ldots, v_{2s-1}v_{2s}b\}$.  
\begin{claim}
\label{vijk}
$v_iv_jv_k \not \in E(H)$ for any $i, j, k \in \{1,2,\ldots,2s\}$. Thus, every hyperedge in the subhypergraph induced by $V(T)$ must contain $b$.
\end{claim}
\begin{proof}
Indeed, if $v_i, v_j, v_k$ belong to three different hyperedges of $T$, then it is easy to find a linear cycle, so suppose two of them belong to the same hyperedge. Without loss of generality, let $\{v_i, v_j\}= \{v_1,v_2\}$. Then replacing $v_1v_2b$ with $v_iv_jv_k$ we can produce a linear path of length 3 in $H$ (here we used that $T$ contains at least 3 hyperedges), a contradiction to the assumption of Section \ref{kleq2}; proving the claim.
\end{proof} 

Now we consider two subcases.

\begin{namedthm*}{Case 1.1}
	\label{oneismissing}
	There exist $i', j' \in \{1,2,3,\ldots,2s\}$ such that $v_{i'}bv_{j'} \not \in E(H)$.
\end{namedthm*}

By Claim \ref{vijk}, every hyperedge in the subhypergraph induced by $V(T)$ must be of the form $v_ibv_j$ for some $i, j \in \{1,2,3,\ldots,2s\}$. 
Therefore, as $v_{i'}bv_{j'} \not \in E(H)$, the degree $d_T(v_{i'})$ of $v_{i'}$ in the subhypergraph induced by $V(T)$ is at most $\abs{V(T)}-3$. Thus by Corollary \ref{outside_degree}, the degree of $v_{i'}$ is at most $\abs{V(T)}-3 +\abs{V(H) \setminus V(T)}+1 = \abs{V(H)}-2$, and we are done.

\begin{namedthm*}{Case 1.2}
	\label{nothing_is_missing}
	$v_ibv_j \in E(H)$ for every $i, j \in \{1,2,3,\ldots,2s\}$.
\end{namedthm*}

Consider a vertex $v_i$ with $i \in \{1,2,3,\ldots,2s\}$. By Claim \ref{vijk}, degree of $v_i$ in the subhypergraph induced by $V(T)$ is $\abs{V(T)}-2$.

Note that there is no hyperedge of the form $v_ixy$ where $x, y \in V(H) \setminus V(T)$ because of the maximality of $T$. Moreover, there is no hyperedge of the form $v_iv_jx$ where $x \in V(H) \setminus V(T)$ and $j \in \{1,2,3,\ldots,2s\}$. Indeed, the hyperedges $v_iv_jx$, $v_ibv_{i'}$, $v_jbv_{j'}$ for any two distinct vertices $i', j'$ with $i', j' \in \{1,2,3,\ldots,2s\} \setminus \{i, j\}$ form a linear cycle. Therefore, any hyperedge containing $v_i$, and a vertex $x \in V(H) \setminus V(T)$, must be of the form $v_ixb$. 

Therefore the total degree of $v_i$ is at most $(\abs{V(T)}-2) + \abs{V(H) \setminus V(T)} = \abs{V(H)}-2$, as desired.

\begin{namedthm*}{Case 2}
	\label{exactly2}
	$T$ consists of exactly 2 hyperedges.
\end{namedthm*}
Let $E(T) =\{a_1a_2b, c_1c_2b \}$. Since $\abs{V(H)} \ge 6$, $\abs{V(H) \setminus V(T)} \not = \emptyset$. We consider the following two subcases.
\begin{namedthm*}{Case 2.1}
	There is no vertex in $V(H) \setminus V(T)$ which is strongly associated to any hyperedge of $T$.
\end{namedthm*}

First suppose $\abs{V(H)} \le 7$. Consider a vertex $v \in V(H) \setminus V(T)$. It is easy to see that  if $vxy$ is a hyperedge, then $x, y \in V(T)$. Moreover, $x,y$ are contained in a hyperedge of $T$, so $v$ is associated to a hyperedge of $T$. Since $v$ is not strongly associated to any hyperedge of $T$ and $T$ has only 2 hyperedges, it follows that $v$ has degree at most $2 \le \abs{V(H)}-2$ (since $\abs{V(H)} \ge 6$) as required.

So we can assume  $\abs{V(H)} \ge 8$. Suppose there is no hyperedge of the form $vxy$ with $v \in V(H) \setminus V(T)$ and $x, y \in V(T)$. Then it is easy to see that any hyperedge of $H$ which contains $a_1 \in V(T)$ is contained in $V(T)$, so degree of $a_1$ is at most $6 \le \abs{V(H)}-2$, as desired. So we can assume that there exists a hyperedge $vx'y'$ with $v \in V(H) \setminus V(T)$ and $x', y' \in V(T)$. It follows that $x', y'$ are contained in a hyperedge of $T$. By assumption $v$ is not strongly associated to any hyperedge of $T$.  So the number of hyperedges $vxy$ such that $x, y \in V(T)$ is at most $2$ (as there are only two hyperedges in $T$). 

Now we upper bound the number of hyperedges $vxy$ such that $x, y \in V(H) \setminus V(T)$. Let us define a (trace) graph $G_v$ on the vertex set $V(G_v) := V(H) \setminus (V(T) \cup \{v\})$ where $ab \in E(G_v)$ if and only if $abv \in E(H)$. Now notice that if there are two edges $pq, rs \in E(G_v)$ that are disjoint then $vpq, vrs, vx'y' \in E(H)$ form a skeleton with 3 hyperedges which contradicts the assumption of Case 2. So every two edges of $G_v$ have a common vertex and so $E(G_v)$ is either a triangle or a star. In either case, $\abs{E(G_v)} \le \abs{V(G_v)}$. 

Therefore, the total degree of $v$ is at most $2 + \abs{E(G_v)} \le 2 + \abs{V(G_v)} = 2 + n - 6 = n - 4$, as desired.

\begin{namedthm*}{Case 2.2}
	There is a vertex $v \in V(H) \setminus V(T)$ which is strongly associated to a hyperedge of $T$.
\end{namedthm*}
As before, notice that if $vxy$ is a hyperedge, then either $x, y \in V(T)$ or $x, y \in V(H) \setminus V(T)$. Assume without loss of generality that $v$ is strongly associated to $a_1a_2b$. So $vba_i \in E(H)$ for some $i \in \{1,2\}$ which implies that there is no hyperedge $vxy$ with $x, y \in V(H) \setminus V(T)$ because otherwise $vxy, vba_i, bc_1c_2$ form a linear path of length 3, a contradiction. Therefore, any hyperedge incident on $v$ is of the form $vxy$ with $x, y \in V(T)$. Moreover, the pair $xy$ is contained in a hyperedge of $T$.

If $v$ is not strongly associated to $bc_1c_2$, then the degree of $v$ is at most $1 + 3 = 4$ and we are done since we assumed $\abs{V(H)} \ge 6$. Therefore, we may assume $v$ is strongly associated to $bc_1c_2$ as well and so $vbc_j \in E(H)$ for some $j \in \{1,2\}$. Now it is easy to see that $a_1a_2c_k \not \in E(H)$ and $c_1c_2a_k \not \in E(H)$ for any $k \in \{1,2\}$ because otherwise we have a linear cycle. If there is a vertex among $\{a_1,a_2,c_1,c_2\}$ with degree at most $2$ in the subhypergraph induced by $V(T)$, then by Corollary \ref{outside_degree}, the degree of this vertex in $H$ is at most $2 + t+ 1 = t+3$ where $\abs{V(H) \setminus V(T)} = t$ but then we are done because $\abs{V(H)} = t+5$. So we may assume that all of the vertices $\{a_1,a_2,c_1,c_2\}$ have degree at least $3$ in the subhypergraph induced by $V(T)$. It is easy to see that the only way this degree condition is met for the vertex $a_i$ is if $a_ibc_1, a_ibc_2 \in E(H)$ for each $i \in \{1,2\}$. This implies that $a_1a_2v, c_1c_2v \not \in E(H)$ because otherwise we have a linear cycle. So the degree of $v$ is at most $4$ and we are done because $\abs{V(H)} \ge 6$.

\begin{namedthm*}{Case 3}
	\label{exactly2}
	$T$ consists of only one hyperedge.
\end{namedthm*}

Let $E(T) = \{abc\}$. Consider the trace graph $G_a$ where $\{x,y\} \in E(G_a)$ if and only if $axy \in E(H)$. Now notice that if there are two edges $pq, rs \in E(G_a)$ that are disjoint then $apq, ars \in E(H)$ form a skeleton with two hyperedges, a contradiction. So every two edges of $G_a$ have a common vertex. It is easy to see that the set of edges of such a graph is either a star (a graph where all the edges have a common vertex) or a triangle. Notice that there may be some isolated vertices in the graph. Since $\abs{V(G_a)} = \abs{V(H)}-1 \ge 5$, it is easy to see that $\abs{E(G_a)} \le \abs{V(G_a)}-1$ holds. So the degree of $a$ in $H$ is $\abs{E(G_a)} \le \abs{V(G_a)}-1 = \abs{V(H)}-2$ as desired.

\subsection{Proof of Theorem \ref{degree_condition}}
\label{universal_pair}

Now we prove Theorem \ref{degree_condition} using Lemma \ref{no_universalpair}. If we assume for a contradiction that Theorem \ref{degree_condition} does not hold, then by Lemma \ref{no_universalpair} we know that there are vertices $u, v \in V(H)$ such that $uvx \in E(H)$ for every $x \in V(H)$ whenever $\abs {V(H)} \ge 6$. 


\begin{lemma}
	\label{final_lemma}
Let $H$ be a $3$-uniform linear-cycle-free hypergraph. Let $s \ge -1$ be an integer. If the degree of every vertex in $H$ is at least $\abs{V(H)}+s$ and $\abs{V(H)} \ge 6$, then it has a subhypergraph $H_0$ such that $\abs{V(H_0)} = \abs{V(H)}-4$ and degree of every vertex in $H_0$ is at least $\abs{V(H_0)}+s+2$.
\end{lemma}

\begin{proof}
Let $\abs{V(H)} = n$. Since $s \ge -1$, the degree of every vertex in $H$ is at least $\abs{V(H)}-1$, so by Lemma \ref{no_universalpair}, there exist vertices $u, v \in V(H)$ such that $uvw \in E(H)$ for every $w \in V(H)$. 
\begin{claim}
	\label{yinab}
Suppose $xyu \in E(H)$ where $x, y \in V(H) \setminus \{u,v\}$. If $xab \in E(H)$ where $a, b \in V(H) \setminus \{u,v,x\}$, then $y \in \{a,b\}$.
\end{claim}
\begin{proof}
Suppose by contradiction that $y \not \in \{a,b\}$. Then the hyperedges, $uva$, $xab$ and $xyu$ form a linear cycle, a contradiction. 
\end{proof}

Since degree of $u$ is at least $n+s \ge n-1$, there exists a hyperedge $x_1y_1u$ where $x_1, y_1 \in V(H) \setminus \{u,v\}$. Consider the trace graph $G_{u,v}$ where $\{p,q\} \in E(G_{u,v}) \text{ if and only if either } pqu \in E(H) \text{ or } pqv \in E(H)$. Let the degree of $x_1$ in $G_{u,v}$ be $d$ and let the corresponding edges be $x_1y_1, x_1y_2, \ldots, x_1y_d$. 

First let us assume $d \ge 2$. If $x_1y_iu, x_1y_jv \in E(H)$ where $i \not = j$, then $x_1y_iu, x_1y_jv$ and $uva$ where $a \not \in \{u,v,y_i,y_j,x_1\}$ form a linear cycle. Therefore, either  for every $1 \le i \le d$, $x_1y_iu \in E(H), x_1y_iv \not \in E(H)$ or for every $1 \le i \le d$, $x_1y_iv \in E(H), x_1y_iu \not \in E(H)$. W.l.o.g assume the former. So degree of $x_1$ in $H$ is $d+1$ plus the number of hyperedges $x_1ab$ such that $a, b \in V(H) \setminus \{u,v,x_1\}$. By assumption $x_1$ has degree at least $n+s \ge n-1$. Since $d +1 \le n-3 + 1 = n-2$, there exists a hyperedge $x_1a_1b_1$ where $a_1, b_1 \in V(H) \setminus \{u,v,x_1\}$. By Claim \ref{yinab}, it follows that $y_1, y_2, \ldots, y_d \in \{a_1,b_1\}$, so $d \le 2$.  Thus, $d = 2$ and so if $x_1ab \in E(H)$ where $a, b \in V(H) \setminus \{u,v,x_1\}$, then $\{y_1, y_2\} = \{a,b\}$. So the degree of $x_1$ is at most $d+ 1 +1 = 4$ a contradiction since $n \ge 6$. 

Now we are left with the case when $d = 1$. By Claim \ref{yinab}, if $x_1ab \in E(H)$ where $a, b \in V(H) \setminus \{u,v,x_1\}$, then $y_1 \in \{a,b\}$, so every hyperedge containing $x_1$, except $x_1uv$, is of the form $x_1y_1a \in E(H)$ where $a \in V(H)\setminus \{x_1,y_1\}$. So, $x_1y_1a \in E(H)$ for every $a \in V(H)\setminus \{x_1,y_1\}$ because otherwise degree of $x_1$ is at most $n-2 < n-s$, a contradiction.
Let the subhypergraph induced by $V(H) \setminus \{u,v,x_1,y_1\}$ be $H_0$. 

Consider an arbitrary vertex $a \in V(H_0)$. It is easy to see that if $abu \in E(H)$ for some $b \in V(H_0) \setminus \{a\}$ then the hyperedges $abu, uvx_1, x_1y_1a$ form a linear cycle, a contradiction. Similarly, $abv, abx_1, aby_1 \not \in E(H)$ for any $b \in V(H_0) \setminus \{a\}$. Moreover, there exists no hyperedge in $H$ that contains one vertex from $uv$, one vertex from $x_1y_1$ and the vertex $a$, since this hyperedge together with $uvw, x_1y_1w$ for any $w \in V(H) \setminus \{u,v,x_1,y_1,a\}$ forms a linear cycle, a contradiction. Therefore, the degree of $a$ in $H_0$ is exactly $2$ less than its degree in $H$. So degree of $a$ in $H_0$ is at least $n+s-2 = \abs{V(H)}+s-2= \abs{V(H_0)} +s+2$. Since the vertex $a$ was chosen arbitrarily, the desired lemma follows.
\end{proof}

We will use the following simple corollary obtained by repeated applications of Lemma \ref{final_lemma}.

\begin{corollary}
	\label{contra}
If $H_l$ is a subhypergraph of $H$ where degree of each vertex in $V(H_l)$ is at least $\abs{V(H_l)}+n_l$, where $n_l \ge -1$, then it has a subhypergraph $H_{l+1}$ such that the degree of every vertex in $V(H_{l+1})$ is at least $\abs{V(H_{l+1})}+n_{l+1}$, where $n_{l+1}=n_l+2$ and $\abs {V(H_{l+1})} = \abs{V(H_{l})}-4$.
\end{corollary}

Assume by contradiction that Theorem \ref{degree_condition} does not hold. That is, there is a linear-cycle-free hypergraph $H:= H_1$ where degree of every vertex is at least $\abs{V(H)} -1$ and $\abs{V(H)} \ge 10$. Then by using Corollary~\ref{contra}, there is an $l$ such that $\abs{V(H_l)} \le 5$ and the degree of every vertex in $H_l$ is at least $\abs{V(H_l)} + 3$ (notice that since $\abs{V(H)} = \abs{V(H_1)} \ge 10$, we must have $l \ge 3$), which is impossible.

%
%
%

\section{Concluding Remarks}
\label{concluding_remarks}

The following problems asked by Gy\'arf\'as, Gy\H{o}ri and Simonovits remain open.

\begin{problem}
	Can one describe the structure of $3$-uniform hypergraphs with no linear cycles?
\end{problem}

It is conceivable that one might construct a linear-cycle-free hypergraph by repeatedly adding hyperedges in a certain fashion. For example, if $H$ is a linear-cycle-free hypergraph, then adding two new vertices $u, v$ to $V(H)$ and adding all the hyperedges of the type $uvx$ for $x \in V(H)$ to $E(H)$, will give us another linear-cycle-free hypergraph.

\begin{problem}
	Which results extend to $r$-uniform hypergraphs?
\end{problem}

For $r=4$ the structure of the ``skeleton" seems to be more complicated. It is, however, conceivable that the current methods are useful for this case. In general, the approach of using skeletons seems to be very effective in proving results about linear-cycle-free hypergraphs. It would be interesting to discover more applications of this approach.

\section*{Acknowledgment}

The research of the authors is partially supported by the National Research, Development and Innovation Office – NKFIH, grant K116769. We thank the anonymous referees for reading our paper very carefully and for their valuable suggestions.

\end{document}